\documentclass[11pt]{article}

\usepackage{amsmath}
\usepackage{amsthm}
\usepackage{amsfonts}
\usepackage{amssymb}
\usepackage{latexsym}
\usepackage{cite}

\usepackage[noautoscale]{youngtab}

\usepackage{fullpage}

\usepackage{color, colortbl}
\usepackage{array}
\usepackage{multirow}
\usepackage{float}
\usepackage{tikz}
\usepackage{enumitem}

\numberwithin{equation}{section}

\newtheorem{thm}{Theorem}[section]
\newtheorem{lem}[thm]{Lemma}

\newtheorem{cor}[thm]{Corollary}

\newtheorem{rem}[thm]{Remark}
\newtheorem{deff}[thm]{Definition}
\newtheorem{exam}[thm]{Example}

\newtheorem{algorithm}[thm]{Algorithm}

\DeclareMathOperator{\ord}{ord}
\DeclareMathOperator{\lcm}{lcm}

\def\bF{{\mathbb F}}

\usepackage[colorlinks=true,
linkcolor=webgreen,
filecolor=webbrown,
citecolor=webgreen]{hyperref}

\definecolor{webgreen}{rgb}{0,.5,0}
\definecolor{webbrown}{rgb}{.6,0,0}

\usepackage[symbol]{footmisc}

\begin{document}

\title{Cycle structures and parities of linearized permutation polynomials~\footnote{H.\ Bian and S.\ Chen were partially supported by
the Strategic Priority Research Program of the Chinese Academy of
Sciences (No.~XDB0510201) and the NSFC grants (No.~12271511).}~\footnote{D.\ Zheng was partially supported by the NSFC grants (No.~62272148).}}

\author{Huajun Bian\footnote{KLMM, Academy of Mathematics and Systems Science, Chinese Academy of Sciences, Beijing 100190, China. Email: bianhuajun@amss.ac.cn.},
Shaoshi Chen\footnote{KLMM, Academy of Mathematics and Systems Science, Chinese Academy of Sciences, Beijing 100190, China.  Email: schen@amss.ac.cn.} and
Dabin Zheng\footnote{Hubei Province Key Laboratory of Applied Mathematics, Faculty of Mathematics and Statistics, Hubei University, Wuhan 430062, China. Email: dzheng@hubu.edu.cn.}}

\date{\today}

\maketitle

\noindent\textbf{Abstract.}

\qquad

Linearized permutation polynomials over finite fields have important applications in coding theory, cryptography, and combinatorics. In this paper, we give an explicit description of the cycle structure of a general linearized permutation polynomial in terms of the Jordan blocks of its Dickson matrix, thereby providing a Dickson-matrix formulation of a
question raised by Mullen and Vaughan in 1988. Building on the cycle structure, we establish a matrix-theoretic framework to characterize their parity. For finite fields of characteristic 2, we determine all odd linearized permutation polynomials explicitly. For fields of odd characteristic, we derive parity criteria in terms of the eigenvalue structure of the associated Dickson matrix. Based on the parity criteria, we then give an algorithm for deciding the parity of a given linearized permutation polynomial.

\qquad





\section{Introduction}\label{sec1}

Permutation polynomials over finite fields have been extensively studied in the past decades~\cite{MR1429394,MR2358760}, due to their applications in coding theory~\cite{MR2284666,MR3154016}, cryptography~\cite{MR4216654}, finite geometry~\cite{MR4847599}, combinatorial design~\cite{MR2259075}, and many other areas. 
A permutation polynomial $f(x)$ over a finite field $\bF_{q^n}$ is called a \emph{linearized permutation polynomial} if it is of the form
\[
f(x)=\sum\limits_{i=0}^{n-1}a_ix^{q^i},
\quad \text{where $a_i \in \bF_{q^n}$}.
\]
It induces a $\bF_q$-linear map of $\bF_{q^n}$, which is viewed as a linear space over $\bF_q$.  The study of linearized permutation polynomials has been motivated by their appearance as linear maps in the construction of MRD codes~\cite{MR514618,MR791529} and as building blocks in the theory of planar functions~\cite{MR226486,MR1432296}.

The cycle structure of a permutation polynomial $f(x)$, viewed as a permutation of $\bF_{q^n}$,  describes the decomposition of $\bF_{q^n}$ into orbits under the action of $f(x)$. It determines, in particular, the compositional order and the parity of the induced permutation. Explicit cycle decompositions are known  for only a few families. For the simplest case when $f(x)$ is a monomial $x^k$, the problem reduces to power maps on the cyclic group $\bF_{q^n}^{*}$ as investigated by Ahmad~\cite{Ahmad1969}. Mullen and Vaughan determined the cycles of linearized permutation polynomials with coefficients in the base field $\bF_q$~\cite{MR959697} in 1988, and inspired several subsequent studies. Lidl and Mullen treated Dickson permutation polynomials \cite{MR1159877}. Subsequent studies considered broader families of permutation functions \cite{MR2966683}, Rédei permutations \cite{MR4412709}, and permutations whose cycle structures are invariant on suitable affine subspaces\cite{MR4149368}.  For instance, Zeng et al. in~\cite{MR4570366} determined the cycle structure of several classes of permutation polynomials over quadratic extensions of finite fields. They subsequently studied the permutation binomials 
\[
  f(x)=\theta x^{2^m+2}+x^{2^{m+1}+1}
\]
over $\bF_{2^{2m}}$ and obtained their complete cycle structure by analyzing the periods of certain auxiliary sequences~\cite{MR4948432}. These results illustrate that the cycle decomposition of a sparse permutation polynomial may depend delicately on its parameters and on arithmetic properties of the extension degree.

Linear maps over finite fields have also been studied as finite dynamical systems. Elspas described the cycles of linear sequential networks over prime fields in terms of the elementary divisors of their transition matrices \cite{Elspas1959}. Hernández Toledo later placed these results in a general framework over arbitrary finite fields, and its periodic behavior can be recovered from the elementary divisors of the associated matrix \cite{HernandezToledo2005}. For an invertible linear map, the functional graph consists entirely of cycles, so this theory provides a general method for determining its cycle structure.

A more specialized description is available when the coefficients of a linearized permutation polynomial lie in the base field $\mathbb F_q$. In this case, $f$ is a polynomial in the Frobenius operator $T:x\mapsto x^q$. Mullen and Vaughan exploited the $T$-invariant decomposition of $\mathbb F_{q^n}$ to determine the cycle structure and to characterize cycle equivalence and strong cycle equivalence for such polynomials~\cite{MR959697}. They also raised the question of extending this analysis beyond the base-field coefficient case. In the present paper, we give a formulation that applies to arbitrary linearized permutation polynomials: using the similarity between the induced $\mathbb F_q$-linear transformation and its Dickson matrix, we express the number of cycles of each length explicitly in terms of the Jordan blocks of the Dickson matrix. This provides a Dickson-matrix treatment of the question considered by Mullen and Vaughan.

The parity of the induced permutation is another principal object of this paper. The sign of an algebraically defined permutation may encode nontrivial arithmetic information, as exemplified by Zolotarev's interpretation of the Legendre symbol \cite{Zolotareff1872}. Moreover, parity phenomena play an important role in the structural and enumerative theory of Latin squares \cite{MR1309160}, providing further motivation for studying the parity of permutations over finite fields. Starting from the cycle structure obtained through the Dickson matrix, we derive explicit parity criteria for linearized permutation polynomials. In characteristic two, we determine all odd linearized permutation polynomials. In odd characteristic, we characterize the parity in terms of the multiplicative orders of the eigenvalues of the Dickson matrix and their algebraic multiplicities. We then turn these criteria into algorithms for deciding whether a given linearized permutation polynomial is even or odd; when all coefficients lie in $\mathbb F_q$, the algorithm can be carried out directly in $\mathbb F_q[x]/(x^n-1)$, without explicitly constructing the Dickson matrix.

The structure of this paper is arranged as follows. In Section~\ref{sec2}, we recall some basic terminologies and  results on linearized permutation polynomials. In Section~\ref{sec3}, we derive the cycle-counting formula in terms of the
Jordan blocks of the Dickson matrix. In Section~\ref{sec4}, we give criteria for determining the parity of a linearized permutation polynomial. In Section~\ref{sec5}, we offer an algorithm for testing
parity and an illustrative example.

\section{Linearized polynomials and Dickson matrices}\label{sec2}

This section recalls basic terminology about linearized polynomials and Dickson matrices from~\cite{MR1429394}. Throughout this paper, let $q$ be a power of some prime number $p$ and $\bF_q$ be the finite field with $q$ elements. Any polynomial $f(x) \in\bF_q[x]$ induces a map from $\bF_q$ to $\bF_q$ given by $c \mapsto f(c)$. Conversely, every map $\bF_q\to\bF_q$ can be represented by a unique polynomial with degree less than or equal to $q-1$, and two polynomials induce the same map if and only if they are congruent modulo $x^q-x$. A polynomial $f(x)\in\bF_q[x]$ is called a \emph{permutation polynomial} if the induced map is a bijection.
\begin{deff}[Linearized polynomials]
Let $n$ be a positive integer. A polynomial $f\in\bF_{q^n}[x]$ is $q$-linearized if it has the form
\begin{equation}\label{eq:linearized}
\sum\limits_{i=0}^{n-1}a_ix^{q^i},  \quad \text{where $a_i \in \bF_{q^n}$}.
\end{equation}
Every polynomial of the form~\eqref{eq:linearized} induces an $\mathbb F_q$-linear map on $\mathbb F_{q^n}$; conversely, every such linear map has a unique representation of the form~\eqref{eq:linearized}. Such a polynomial is also called a $q$-polynomial. If it induces a permutation of $\mathbb F_{q^n}$, we call it a $q$-linearized permutation polynomial.
\end{deff}
\begin{exam}\label{exam2.2}
$f(x)=x^9+x^3-x\in\bF_{27}[x]$ is a 3-linearized permutation polynomial of $\bF_{27}$.
\end{exam}
To each $q$-polynomial $f(x)=\sum\limits_{i=0}^{n-1}a_ix^{q^i}$,  Dickson~\cite{Dickson1896} introduced the associated Dickson matrix
\begin{align}\label{1.1}
D_f =
    \begin{pmatrix}
        a_0            & a_1            & a_2            & \cdots & a_{n-1}            \\
        a_{n-1}^{q}    & a_0^{q}        & a_1^{q}        & \cdots & a_{n-2}^{q}        \\
        \vdots         & \vdots         & \vdots         & \ddots & \vdots             \\
        a_1^{q^{n-1}}  & a_2^{q^{n-1}} & a_3^{q^{n-1}} & \cdots & a_0^{q^{n-1}}
    \end{pmatrix}.
\end{align}
The basic properties of Dickson matrices are summarized in the following theorem.
\begin{thm}\label{thm2.3}
Let $f(x)$ be a $q$-linearized polynomial over $\bF_{q^n}$ with its associated Dickson matrix $D_f$. Then
\begin{enumerate}[label=\upshape(\roman*)]
\item $f(x)$ is a permutation polynomial if and only if $D_f$ is invertible.
\item $D_f$ is $\bF_{q^n}$-similar to the matrix of the $\bF_q$-linear map $f: \bF_{q^n} \to \bF_{q^n}$ with respect to any basis of $\bF_{q^n}$ over $\bF_q$.
\end{enumerate}
\end{thm}
\begin{proof}
The first assertion was proved by Dickson in his thesis~\cite{Dickson1896} (one can also see~\cite[page 362]{MR1429394}).
The second assertion is Lemma 4.1 in~\cite{MR3044095}.
\end{proof}

Let $A\in GL_n(\bF_q)$ be an invertible matrix, since the characteristic polynomial of $A$ splits over a finite extension of $\bF_q$, all eigenvalues of $A$
belong to $\bF_{q^m}$ for some positive integer $m$. In particular, the multiplicative order of each eigenvalue is well defined. We define the order set of $A$ by $S(A):=\left\{\ord(\lambda):\lambda \text{ is an eigenvalue of } A\right\}$, where $\ord(\lambda)$ denotes the multiplicative order of $\lambda$. We further define the totalized order set of $A$ by $T(A):=\left\{\lcm(E):E\subseteq S(A)\right\}$, where the least common multiple of a finite set of positive integers is defined to be the least common multiple of all elements of it, with the convention that $\lcm(\emptyset)=1$. 

\begin{lem}\label{lem2.4}
Every element of $S(A)$ is relatively prime to $p$. Moreover, $T(A)$ is closed under the least common multiple operation.
\end{lem}

\begin{proof}
Let $\lambda$ be an eigenvalue of $A$. Since $\lambda\in\bF_{q^m}$ for some $m$, it belongs to the finite cyclic group $\bF_{q^m}^{\times}$, whose order is $q^m-1$. Hence $\ord(\lambda)\mid(q^m-1)$. Because $q^m-1$ is relatively prime to $p$, every element of $S(A)$ is relatively prime to $p$.

To prove the second assertion, let $r_1,r_2\in T(A)$. By definition, there exist subsets $E_1,E_2\subseteq S(A)$ such that $r_1=\lcm(E_1)$ and $r_2=\lcm(E_2)$. Therefore, 
\[
\lcm(r_1,r_2)=\lcm\bigl(\lcm(E_1),\lcm(E_2)\bigr)=\lcm(E_1\cup E_2).
\]
Since $E_1\cup E_2\subseteq S(A)$, it follows that $\lcm(r_1,r_2)\in T(A)$. Thus $T(A)$ is closed under the least common multiple operation.
\end{proof}

Throughout the paper, we use the notation
\[
S:=S(D_f)=\{\theta_1,\ldots,\theta_s\}
\]
and write
\[
T:=T(D_f)
\]
for short. For each $i\in\{1,\ldots,s\}$, define
\[
u_i
:=
\sum_{\substack{\lambda\text{ is an eigenvalue of }D_f\\
                 \ord(\lambda)=\theta_i}}
\operatorname{mult}_{D_f}(\lambda),
\]
where $\operatorname{mult}_{D_f}(\lambda)$ denotes the algebraic multiplicity of $\lambda$ as an eigenvalue of $D_f$. We call $u_i$ the order weight of $\theta_i$.

\begin{exam}
Let $f(x)=x^9+x^3-x\in\bF_{27}[x]$. The associated Dickson matrix is
\[
D_f=
\begin{pmatrix}
-1 & 1 & 1\\
1 & -1 & 1\\
1 & 1 & -1
\end{pmatrix}.
\]
The characteristic polynomial is
\begin{align*}
\det(\lambda E_3-D_f)
&= \det\begin{pmatrix} 1+\lambda & -1 & -1 \\ -1 & 1+\lambda & -1 \\ -1 & -1 & 1+\lambda \end{pmatrix} \\
&= (\lambda - 1)^3.
\end{align*}
Hence $D_f$ has a unique eigenvalue $\lambda = 1$ with algebraic multiplicity $3$.
The set of distinct orders of eigenvalues of $D_f$ is $\{\theta_1\}=\{1\}$ and the corresponding total algebraic multiplicity is $u_1=3$.

\end{exam}

\section{Cycles of linearized permutation polynomials}\label{sec3}

In this section, we investigate the cycle structure of a linearized permutation polynomial $f(x)$ over $\bF_{q^n}$. By Theorem~\ref{thm2.3}, the linear map induced by $f(x)$ is $\bF_{q^n}$-similar to its Dickson matrix $D_f$. This allows us to translate the cycle-counting problem into a question about the Jordan normal form of $D_f$ over the algebraic closure of $\bF_q$.

Let $f(x)$ be a linearized permutation polynomial. For each $a\in\bF_{q^n}$, define $m_a$ to be the minimal positive integer such that $f^{(m_a)}(a)=a$, where $f^{(k)}$ denotes the $k$-th composition of the map induced by $f(x)$. For a positive integer $i$, we define
\[K_i:=\{a\in\bF_{q^n}:m_a=i\}\quad \text{ and } \quad C_i:=\frac{|K_i|}{i},\]
where $|K_i|$ denotes the cardinality of the set $K_i$. 
Then $C_i$ counts the number of cycles with length $i$ in the decomposition of $f$. Since $a \in \bF_{q^n}$ is fixed by $f^{(k)}$ if and only if $m_a \mid k$, we have
\begin{align}\label{eq:fixpoint}
|\{a\in\bF_{q^n}:f^{(k)}(a)=a\}|=\sum_{i \mid k} |K_i|.
\end{align}
On the other hand, since $f$ induces a linear transformation, the map $f^{(k)} - \text{id}$ is also $\bF_q$-linear, and the fixed points set of $f^{(k)}$ is precisely $\ker(f^{(k)} -\text{id})$, which is an $\bF_q$-subspace of $\bF_{q^n}$ with cardinality $q^{\dim(\ker(f^{(k)} - \text{id}))}$. By Theorem~\ref{thm2.3}, the matrix of $f$ is similar to the Dickson matrix $D_f$, hence the matrix of $f^{(k)} - \text{id}$ is similar to $D_f^k - E_n$, where $E_n$ is the $n\times n$ identity matrix. Combining these observations yields
\begin{align}\label{3.2}
\sum\limits_{i\mid k}|K_i|=q^{\dim(\ker(D_f^k - E_n))}.
\end{align}

To evaluate the right hand side of (\ref{3.2}), we analyze the Jordan normal form of $D_f^k-E_n$ over the algebraic closure of $\bF_{q^n}$. Any square matrix is similar to a block diagonal matrix with each block is of the form 
$J_t(\lambda) = \lambda E_t + N_t$,  where $\lambda$ is an eigenvalue of $D_f$ and
\[N_t=\begin{pmatrix}
0&1&&\\
&0&\ddots&\\
&&\ddots&1\\
&&&0
\end{pmatrix}_{t\times t}.\]
Since $\operatorname{ker}(D_f^k - E_n)$ is the direct sum of the kernels of all corresponding Jordan blocks, it suffices to compute the kernel dimension for each block.
\begin{lem}\label{lem3.1}
Let $J_t(\lambda) = \lambda E_t + N_t$ be a Jordan block of order $t$ with eigenvalue $\lambda$, and $p$ be the characteristic of $\bF_q$. For any positive integer $k$,
\[\dim_{\overline{\bF_q}}(\ker(J_t(\lambda)^k - E_t)) =
\begin{cases}
\min(t, p^{v_p(k)}) & \text{if } \lambda^k = 1, \\
0 & \text{if } \lambda^k \neq 1,
\end{cases}\]
where $v_p(k)$ is the $p$-adic valuation of $k$,  i.e., the largest integer $m$ such that $p^m$ divides $k$.
\end{lem}
\begin{proof} We first consider the case when $\lambda^k \neq 1$. 
Since the only eigenvalue of $J_t(\lambda)$ is $\lambda$, $J_t(\lambda)^k - E_t$ has a unique eigenvalue $\lambda^k-1\neq 0$. Thus $J_t(\lambda)^k - E_t$ is invertible and its kernel is trivial.

We next consider the case when $\lambda^k = 1$.
Since $N_t^t = 0$ and $\lambda E_t$ commutes with $N_t$, we have
\[J_t(\lambda)^k - E_t = {(\lambda E_t + N_t)}^k - E_t = \lambda^kE_t+\sum\limits_{j=1}^{t-1} \binom{k}{j} \lambda^{k-j} N_t^j-E_t=\sum\limits_{j=1}^{t-1} \binom{k}{j} \lambda^{k-j} N_t^j.\]
Thus $J_t(\lambda)^k - E_t$ is strictly upper triangular, with the $j$-th superdiagonal entries all equal to $\binom{k}{j}\lambda^{k-j}$:
\[
J_t(\lambda)^k - E_t =
\begin{pmatrix}
0 & \binom{k}{1}\lambda^{k-1} & \binom{k}{2}\lambda^{k-2} & \cdots & \binom{k}{t-1}\lambda^{k-t+1} \\
0 & 0 & \binom{k}{1}\lambda^{k-1} & \cdots & \binom{k}{t-2}\lambda^{k-t+2} \\
0 & 0 & 0 & \ddots & \vdots \\
\vdots & & & \ddots & \binom{k}{1}\lambda^{k-1} \\
0 & 0 & \cdots & 0 & 0
\end{pmatrix}.
\] By Lucas' theorem on binomial coefficients~\cite{Lucus}, when $v_p(j) < v_p(k)$, writing
\[
k = \sum_{i \ge v_p(k)} k_i p^i, \qquad j = \sum_{i \ge v_p(j)} j_i p^i,
\]
we obtain
\[
\binom{k}{j} \equiv \prod_{i} \binom{k_i}{j_i}
\equiv \binom{0}{j_{v_p(j)}} \prod_{i > v_p(j)} \binom{k_i}{j_i}
\equiv 0 \pmod{p},
\]
since the digit of $k$ at position $v_p(j)$ is $0$ , while the digit of $j$ at that position is $j_{v_p(j)}>0$, forcing $\binom{k_{v_p(j)}}{j_{v_p(j)}} = \binom{0}{j_{v_p(j)}} = 0$. In particular, when $1 \le j < p^{v_p(k)}$, we have $v_p(j) < v_p(k)$, so 
\[
\binom{k}{j} \equiv 0 \pmod{p}.
\]

For $j = p^{v_p(k)}$, Lucas' theorem gives
\[
\binom{k}{j} \equiv \binom{k_{v_p(k)}}{1} \not\equiv 0 \pmod{p}.
\]
Thus $j = p^{v_p(k)}$ is the smallest positive integer $j$ for which $\binom{k}{j} \not\equiv 0 \pmod{p}$. If $p^{v_p(k)}<t$, then the leading nonzero superdiagonal is $p^{v_p(k)}$-th. So the kernel has dimension $p^{v_p(k)}$.
If $p^{v_p(k)}\geq t$, then all coefficients $\binom{k}{j}$ are $0$ in $\bF_q$ for $1\leq j\leq t-1$, whence $J_t(\lambda)^k - E_t=0$, and the kernel has dimension $t$.
Combining both cases gives that $\operatorname{dim}\operatorname{ker}({J_t(\lambda)}^k-E_t)$ is equal to $\min(t, p^{v_p(k)})$ when $\lambda^k = 1$.
\end{proof}
We are now ready to describe the cycle structure of $f(x)$ as a permutation of $\bF_{q^n}$.  Recall that each element in $T$ is either $1$ or the least common multiple of the orders of some eigenvalues
\begin{thm}[Cycle structure]
\label{thm:cycle}
Let $f$ be a linearized permutation polynomial over $\bF_{q^n}$ with Dickson matrix $D_f$. Let $\mathcal{J}(D_f)$ be the multiset of Jordan blocks of $D_f$. Then for each positive integer $\ell$, the number of cycles of length $\ell$ in the permutation induced by $f$ on $\bF_{q^n}$ is
\[
    C_\ell= \frac{1}{\ell} \sum_{d \mid \ell} \mu\!\left(\frac{\ell}{d}\right)q^{\,\beta(d)},
\]
where $\mu$ is the classical Möbius function, and
\[
    \beta(d)
    = \sum_{\substack{J_{t_\lambda}(\lambda) \in \mathcal{J}(D_f)\\ \mathrm{ord}(\lambda) \mid d}}
      \min\!\bigl(t_\lambda,\, p^{v_p(d)}\bigr),
\]
In particular, we have $C_\ell = 0$ unless $\frac{\ell}{p^{v_p({\ell})}} \in T$.
\end{thm}

\begin{proof}
Lemma~\ref{lem3.1} together with \eqref{eq:fixpoint} gives $\sum_{i \mid k}|K_i| = q^{\beta(k)}$ for all $k \geq 1$. Setting $F(k) = \sum_{i \mid k}|K_i|$, the classical Möbius inversion
formula yields
\[
    |K_\ell|
    = \sum_{d \mid \ell} \mu\!\left(\frac{\ell}{d}\right) F(d)
    = \sum_{d \mid \ell} \mu\!\left(\frac{\ell}{d}\right) q^{\beta(d)},
\]
and dividing by $\ell$ gives $C_\ell = |K_\ell|/\ell$.

When $\frac{\ell}{p^{v_p({\ell})}} \notin T$, Take $\tilde{\ell} = \lcm\{t \in T | t \mid \ell\}$. Since $T$ is the set of least common multiples of subsets of $S$ and $\lcm(\lcm(A), \lcm(B)) = \lcm(A \cup B)$ it follows that $\tilde{\ell} \in T$. In particular this forces $\frac{\ell}{v_p(\ell)} \neq \tilde{\ell}$. So there exists a prime $p' \neq p$ dividing $\frac{\ell}{\tilde{\ell}}$. Write $e_1 := v_{p'}(\ell)$.

We claim that for every $d \mid l$ with $v_{p'}(d) = e_1 - 1$,
\[
\beta(d) = \beta(dp'),
\qquad
\mu\!\left(\frac{l}{d}\right) = -\,\mu\!\left(\frac{l}{dp'}\right).
\]

Since $v_{p'}(d) = e_1 - 1$ and $v_{p'}(\ell) = e_1$, we have
$v_{p'}\!\left(\frac{\ell}{d}\right) = 1$, so
$\frac{\ell}{d} = p' \cdot \frac{\ell}{dp'}$ with
$p' \nmid \frac{\ell}{dp'}$. By the multiplicativity of the M\"obius
function on squarefree arguments,
\[
\mu\!\left(\frac{l}{d}\right)
= \mu(p') \, \mu\!\left(\frac{l}{dp'}\right)
= -\,\mu\!\left(\frac{l}{dp'}\right).
\]

Since $d \mid dp'$, $\ord(\lambda) \mid d$ implies $\ord(\lambda) \mid dp'$. Conversely, suppose $\ord(\lambda) \mid dp'\mid \ell$. Since $\ord(\lambda)\in T$, we have $\ord(\lambda)\mid\tilde{\ell}$, thus $v_{p'}(\ord(\lambda))\leq v_{p'}(\tilde{\ell})<e_1$. So $\ord(\lambda) \mid d$. On the other hand, since $\gcd(p,p')=1$, $v_p(d)=v_p(dp')$.  Hence $\beta(d)=\beta(dp')$. Thus
\begin{align*}
|K_\ell|
&= \sum_{d \mid \ell} \mu\!\left(\frac{\ell}{d}\right) q^{\beta(d)}\\
&=\sum_{\substack{d \mid \ell \\ v_{p'}(d) = e_1 - 1}} \mu\!\left(\frac{\ell}{d}\right) q^{\beta(d)}+\sum_{\substack{d \mid \ell \\ v_{p'}(d) = e_1}} \mu\!\left(\frac{\ell}{d}\right) q^{\beta(d)}\\
&=\sum_{\substack{d \mid \ell \\ v_{p'}(d) = e_1 - 1}} \left(\mu\!\left(\frac{\ell}{d}\right) q^{\beta(d)}+\mu\!\left(\frac{\ell}{dp'}\right) q^{\beta(dp')}\right)\\
&=\sum_{\substack{d \mid \ell \\ v_{p'}(d) = e_1 - 1}} \left(\mu\!\left(\frac{\ell}{d}\right) q^{\beta(d)}-\mu\!\left(\frac{\ell}{d}\right) q^{\beta(d)}\right)\\
&=0
\end{align*}
\end{proof}
Thus Theorem~3.2 may be viewed as the Dickson-matrix form of the general cycle-counting formula for invertible linear finite dynamical systems.
\begin{rem}
Assume that the coefficients of $f$ lie in $\bF_q$. Let $T\colon x\mapsto x^q$ be the Frobenius operator on $\bF_{q^n}$. Then there exists a polynomial $\tilde f\in\bF_q[X]$ such that $f(x)=\tilde{f}(T)(x)$. With respect to a normal basis of $\bF_{q^n}$ over $\bF_q$, the Dickson matrix $D_f$ is the matrix of $\tilde{f}(T)$. 

Let $g$ be an irreducible factor of $X^n-1$, and, for $a\geq 1$, put  $W^{(a)}=\ker(g(T)^a)$. The chain $W^{(1)}\subseteq W^{(2)}\subseteq\cdots$ stabilizes, define $W$ to be the stable space. The restriction of $D_f$ to the $W$ has eigenvalues $\tilde{f}(\omega)$, where $\omega$ runs through the roots of $g$. These eigenvalues are Galois conjugate and therefore have the same multiplicative order $j$.

Moreover, ~\cite{MR959697} gives the fixed-point spaces of $f^{jp^a}$ restricted to $W$ is $\ker(g(T)^{p^as})\cap W$, where $g(x)^s\parallel \tilde{f}(x)^j-1$, its cardinality is $q^{\beta(d)}$, in agreement with Theorem~\ref{thm:cycle}.
\end{rem}

\begin{exam}
Let $f(x)=x^9+x^3-x\in\bF_{27}[x]$. The associated Dickson matrix is
\[
D_f=
\begin{pmatrix}
-1 & 1 & 1\\
1 & -1 & 1\\
1 & 1 & -1
\end{pmatrix}.
\]
The Jordan normal form of $D_f$ is
\[
J_{D_f}=
\begin{pmatrix}
1 & 1 & 0\\
0 & 1 & 0\\
0 & 0 & 1
\end{pmatrix}.
\]
Consequently, $\beta(d)$ in Theorem~\ref{thm:cycle} is given by
\[
\beta(d)=
\begin{cases}
2, & \text{if } v_3(d)=0,\\[2mm]
3, & \text{if } v_3(d)\geq 1.
\end{cases}
\]
By Theorem~\ref{thm:cycle}, $C_l=0$ whenever $l$ is not a power of $3$.
It remains to consider the case $l=3^i$.
For $i=0$, we have
\[C_1=1\sum_{d\mid 1}\mu(\frac{1}{d})3^{\beta(d)}=1\mu(1)3^{\beta(1)}=9.\]
For $i\geq 1$, we have
\[C_{3^i}=\frac{1}{3^i}\sum_{d\mid 3^i}\mu\left(\frac{3^i}{d}\right)q^{\beta_d}.\]
Since the only divisors of $3^i$ contributing nonzero Möbius values
are $3^{i-1}$ and $3^i$, we have
\[C_{3^i}=\frac{1}{3^i}\left[\mu(3)q^{\beta_{3^{i-1}}}+\mu(1)q^{\beta_{3^i}}\right]=\frac{1}{3^i}\left(-q^{\beta_{3^{i-1}}}+q^{\beta_{3^i}}\right).\]
Therefore,
\[
C_{3^i}=
\begin{cases}
6, & \text{if } i=1,\\[2mm]
0, & \text{if } i\geq 2.
\end{cases}
\]
In summary,  the cycle structure of $f$ consists of $6$
cycles of length $3$ and $9$ cycles of length $1$.
\end{exam}

\section{Parity criteria for linearized permutation polynomials}\label{sec4}

We now apply the argument about cycle structures to determine the parity of the permutations induced by linearized permutation polynomials over $\bF_{q^n}$.

We define the parity of a permutation polynomial to be the parity of the permutation it induces on $\bF_{q^n}$. Recall that every cycle of length $i$ in the disjoint cycle decomposition of a permutation can be written as a product of $i - 1$ transpositions, for example,
\[(a_1\, a_2\, \cdots\, a_i) = (a_1\, a_i)(a_1\, a_{i-1})\cdots (a_1\, a_2).\]
Hence the permutation induced by a linearized permutation polynomial $f$ decomposes into
\begin{equation*}
    \sum_{i=1}^{q^n} C_i (i-1)
    = \sum_{i=1}^{q^n} \frac{|K_i|}{i}(i-1)
    = q^n - \sum_{i=1}^{q^n} \frac{|K_i|}{i}
\end{equation*}
transpositions.
Since $q^n \equiv q \pmod{2}$, the right hand side has the same parity as
\begin{align}
    Par(f):=q + \sum_{i=1}^{q^n} \frac{|K_i|}{i},
\end{align}
so it suffices to determine the parity of $Par(f)$.

\begin{thm}
\label{thm:parity}
Let $f(x) = \sum_{i=0}^{n-1} a_i x^{q^i}$ be a $q$-linearized permutation polynomial over $\bF_{q^n}$, and let $D_f$ denote its associated Dickson matrix, whose distinct eigenvalue orders are
$\theta_1, \ldots, \theta_s$ with corresponding sums of algebraic multiplicities $u_1, \ldots, u_s$, respectively.
\begin{enumerate}[label=\upshape(\roman*)]
    \item If $\operatorname{char}(\bF_q) = 2$,
    then $f$ is an odd permutation if and only if $n = q
    = 2$ and $f(x) = bx^2$ for some $b\neq 0$.
    \item If $\operatorname{char}(\bF_q) \neq 2$,
    then the parity of $f$ coincides with the parity of
    the integer
    \[
        \sum_{j=1}^{s}
        \frac{q^{u_j} - 1}{2^{v_2(\theta_j)}},
    \]
    where $v_2(\theta_j)$ denotes the $2$-adic valuation
    of $\theta_j$.
\end{enumerate}
\end{thm}
The proof naturally splits into two cases according to the characteristic of $\bF_{q^n}$.

\subsection{The case when $\operatorname{char}(\bF_{q^n})=2$}

In this case, $q$ is a power of $2$, hence even. Thus the parity of $Par(f)$ is determined by the parity of
\begin{equation*}
   Par(f)= q + \sum_{i=1}^{q^n} \frac{|K_i|}{i}\equiv \sum_{i=1}^{q^n} \frac{|K_i|}{i} \pmod{2}.
\end{equation*}

By Lemma~\ref{lem2.4}, every element of $T$ is odd. Theorem~\ref{thm:cycle} gives $K_i = \emptyset$ unless $i = 2^t j$ for some $t \geq 0$ and $j \in T$. Consequently,
\begin{align}\label{4.2}
    \sum_{i=1}^{q^n} \frac{|K_i|}{i}\equiv \sum_{t \geq 0} \sum_{j \in T} \frac{|K_{2^t j}|}{2^t j}\equiv \sum_{t \geq 0} \sum_{j \in T} \frac{|K_{2^t j}|}{2^t}\pmod{2},
\end{align}
where the second congruence holds since $j$ is odd.

Let $m_k$ denote the number of Jordan blocks of $D_f$ of size at least $k$, so that $\{m_k\}$ is a non-increasing sequence. Let $r$ be the maximal size of the Jordan blocks of $D_f$, then $m_1\ge m_2\ge \cdots\ge m_r>m_{r+1}=0$. We consider two cases.

\medskip
\noindent Case (i): $r = 1$ (i.e., $D_f$ is diagonalizable).

Here $m_2 = 0$, thus for every even integer $i$ and Jordan block $J$, $J^i-E$ and $J^{i/2}-E$ has the same rank, thus by (\ref{eq:fixpoint}), $K_i = \emptyset$ for all even $i$. Hence
\begin{align*}
\sum_{i=1}^{q^n} \frac{|K_i|}{i}\equiv\sum_{i\in T}\frac{|K_i|}{i}\equiv\sum_{i\in T}|K_i|\equiv q^n\pmod 2.
\end{align*}
and $f$ is an even permutation.

\medskip
\noindent Case (ii): $r > 1$ (i.e., $D_f$ is not diagonalizable).

Let $t$ be the largest integer such that $m_{2^t} > 0$. Expanding (\ref{4.2}) and we have
\begin{align}
\sum_{t\geq 0}\sum_{i\in T}\frac{|K_{{2^t}\cdot i}|}{2^t}
&= \sum_{i\in T}|K_i|+\sum_{t\geq 1}\sum_{i\in T}\frac{|K_{{2^t}\cdot i}|}{2^t}\nonumber\\
&= q^{m_1}+\sum_{j=0}^{t}\frac{1}{2^{j+1}}(q^{m_1+\cdots+m_{2^{j+1}}}-q^{m_1+\cdots+m_{2^j}})\nonumber\\
&\equiv\sum\limits_{j=0}^{t}\frac{q^{m_1+\cdots+m_{2^j}}}{2^{j+1}}(q^{m_{2^j+1}+m_{2^j+2}+\cdots+m_{2^{j+1}}}-1)\nonumber\\
&\equiv\sum\limits_{j=0}^{t-1}\frac{q^{m_1+\cdots+m_{2^j}}}{2^{j+1}}+\frac{q^{m_1+\cdots+m_{2^t}}}{2^{t+1}}(q^{m_{2^t+1}+m_{2^t+2}+\cdots+m_{r}}-1)\pmod 2.\label{4.3}
\end{align}
where the last step uses $q \equiv 0 \pmod{2}$, so $q^k - 1 \equiv 1 \pmod{2}$ for all $k \geq 1$.

We claim that for $2 \leq j \leq t-1$, the term $\dfrac{q^{m_1 + \cdots + m_{2^j}}}{2^{j+1}}$ is even. Indeed, since $\{m_i\}$ is non-increasing and $m_{2^t} \geq 1$,
\[
m_1 + m_2 + \cdots + m_{2^j} \geq 2^j > j+1,
\]
and therefore
\[
v_2\!\left( q^{m_1 + \cdots + m_{2^j}} \right)
\geq v_2\!\left( 2^{m_1 + \cdots + m_{2^j}} \right)
= m_1 + \cdots + m_{2^j} \geq j+1,
\]
so $\dfrac{q^{m_1+\cdots+m_{2^j}}}{2^{j+1}} \in 2\mathbb{Z}$. Hence \eqref{4.3} reduces to
\begin{equation}\label{4.4}
\sum_{i=1}^{q^n} \frac{|K_i|}{i}
\equiv \frac{q^{m_1}}{2} + \frac{q^{m_1+m_2}}{4}(1 - \delta_{m_3,0}) \pmod{2}.
\end{equation}

We further distinguish the following subcases.

\begin{enumerate}[label=\textup{(\alph*)}]
    \item If $q = 2^m$ for some $m > 1$, or if $m_1 > 1$, then
    $\frac{q^{m_1}}{2}$ is even. Moreover, when $m_3 > 0$,
    $\frac{q^{m_1+m_2}}{4}$ is also even.
    Hence the right-hand side of \eqref{4.4} is even and $f$ is an even permutation.

    \item If $q = 2$ and $m_1 = 1$, then since $r > 1$ we have $m_2 = 1$.
    \begin{itemize}
        \item If $m_3 > 0$: the right-hand side of
        \eqref{4.4} equals $\frac{2}{2} + \frac{2^2}{4}
        = 1 + 1 = 2 \equiv 0 \pmod{2}$, so $f$ is even.
        \item If $m_3 = 0$: the right-hand side equals
        $\frac{2}{2} = 1 \equiv 1 \pmod{2}$, so $f$ is odd.
    \end{itemize}
\end{enumerate}

In subcase~(b) with $m_3 = 0$, we have $n = q = 2$ and the minimal polynomial of $D_f$ is $(x - a)^2 = x^2 - a^2$ for some $a \neq 0$. Since this is also the minimal polynomial of the induced linear map
on $\bF_{q^n}$, we have $a^2 \in \bF_2$, forcing $a = 1$. Writing
\[
    D_f = \begin{pmatrix} a_0 & b_0 \\ b_0^2 & a_0^2 \end{pmatrix},
\]
its characteristic polynomial is $\chi_{D_f}(x) = x^2 + (a_0 + a_0^2)x + (a_0^3 + b_0^3)$. For $\chi_{D_f}(x) = x^2 + 1$, we need $a_0 + a_0^2 = 0$, i.e.\ $a_0 \in \bF_2$.
\begin{itemize}
    \item $a_0 = 1$: $f(x) = x$, which is the identity, hence even.
    \item $a_0 = 0$: $f(x) = b_0x^2$, which induces a single transposition
    on $\bF_4$, hence odd.
\end{itemize}

Combining all cases, we conclude:

\begin{cor}\label{cor:p2-parity}
    Let $q$ be a power of $2$. A $q$-linearized permutation polynomial $f$ over $\bF_{q^n}$ is an odd permutation if and only if $n = q = 2$ and $f(x) = bx^2$ for some $b\in\bF_4^*$.
\end{cor}

\subsection{The case when $\operatorname{char}(\bF_{q^n})\neq 2$}
In this case, $q$ is odd. Thus, the parity of $f$ is equivalent to the parity of $Par(f)\equiv 1 + \sum_{i=1}^{q^n} \frac{|K_i|}{i}\pmod 2$. By Lemma~\ref{lem2.4}, every element in $T$ is coprime with $p$. Moreover, by Theorem~\ref{thm:cycle}, we have
\begin{align}
Par(f)\equiv 1 + \sum_{i\geq 1} \frac{|K_i|}{i}\equiv 1+\sum_{i\in T}\sum_{j\geq 0}\frac{\left|K_{p^j\cdot i}\right|}{p^j\cdot i}\equiv\sum_{i\in T}\frac{\displaystyle\sum_{j\geq 0}|K_{p^j\cdot i}\setminus\{0\}|}{i}\pmod 2.\label{eq:parity-odd}
\end{align}
The last identity is because $0$ can only appear in $K_1$, thus $K_i\setminus\{0\}=K_i$ when $i>1$ and $|K_1\setminus\{0\}|=|K_1|-1$.

Set
\[
    \Phi \colon \mathbb{Z} \to 2^S, \quad i \mapsto \{\theta_j \in S : \theta_j \mid i\}.
\]
The key observation is the following.

\begin{lem}\label{lem4.3}
For a positive integer $i$ coprime with $p$, the sum $\displaystyle\sum_{j\geq 0}|K_{p^j\cdot i}\setminus\{0\}|$ depends only on the sequence $\{u_j\}_{1\leq j\leq s}$ and not on the entries of the Jordan form of $D_f$ off the diagonal .
\end{lem}
\begin{proof}
By Theorem~\ref{thm:cycle}, if $i\notin T$, then $K_{p^j\cdot i}$ is empty for all nonnegative $j$. so the sum is $0$.

Now suppose $i\in T$. Write $\Phi(i)=\{\theta_{k_1},\cdots,\theta_{k_r}\}$. We proceed by induction on the size of $\phi(i)$.

\medskip
\noindent\textit{Base case.} If $r=0$, then $i=1$ and the sum is $0$. When $r = 1$, since $i=\theta_1$ is not divisible by other elements in $T$ and $K_j$ is empty when $j>q^n$, we have
\[
\sum_{j\geq 0} |K_{p^j \cdot i} \setminus \{0\}|
= \Bigl|\bigcup_{j\geq 0} K_{p^j \cdot i} \setminus \{0\}\Bigr|
= \Bigl|\bigcup_{j\mid q^ni} K_{j} \setminus \{0\}\Bigr|
= |\ker(D_f^{\,q^n \cdot i} - E) \setminus \{0\}|
= q^{u_{k_1}} - 1,
\]
which depends only on $u_{k_1}$.

\medskip
\noindent\textit{Inductive step.}
Assume the lemma holds for all positive integers whose index set has cardinality at most $r - 1$ ($r\geq 2$). Since $\gcd(i,p)=1$, every divisor of $q^ni$ can be written uniquely in the form $p^jt$, where $j\geq 0$ and $t\mid i$. The divisors with $t=i$ are precisely $p^ji$, while the remaining ones correspond to proper divisors $t\mid i$, $t\neq i$. We therefore have
\begin{align}
\sum_{j\geq 0} |K_{p^j \cdot i} \setminus \{0\}|= \sum_{j \mid q^n \cdot i} |K_j|
   - \sum_{\substack{t \mid i,\, t \neq i}}
     \sum_{j\geq 0} |K_{p^j \cdot t}|= q^{u_{k_1} + \cdots + u_{k_r}}
   - \sum_{\substack{t \mid i,\, t \neq i}}
     \sum_{j\geq 0} |K_{p^j \cdot t}|. \label{eq:induct-sum}
\end{align}
For each $t \mid i$ with $t \neq i$: if $t \notin T$ the inner sum vanishes; if $t \in T$, then $\Phi(t)$ has at most $r - 1$ elements, so the induction hypothesis applies and the inner sum depends only on the $\{u_j\}_{1\leq j\leq s}$. Hence the entire expression \eqref{eq:induct-sum} depends only on $\{u_j\}$, completing the induction.
\end{proof}

Let $D_f^\Delta$ denote a matrix in $M_n(\bF_q)$ similar to the diagonal of the Jordan form of $D_f$, then $D_f^\Delta$ induces a permutation of $\bF_q^n$. For each $a\in\bF_q^n$, define $m_a^\Delta$ to be the minimal positive integer such that ${D_f^\Delta}^{m_a}(a)=a$, $K_i^\Delta:=\{a\in\bF_q^n:m_a^\Delta=i\}$ and $C_i^\Delta=|\frac{K_i^\Delta}{i}$ be the number cycles with length $i$. Then $Par(f)\equiv 1 + \sum_{i\geq 1} \frac{|K_i|}{i}\pmod 2$. Lemma~\ref{lem3.1} gives $|K_i^\Delta| = 0$ whenever $i \notin T$. 

Now we consider the case $i\in T\setminus S$, let $\{\theta_{k_1},\cdots,\theta_{k_r}\}$ denotes $\Phi(i)$. If $\Phi(i)= \emptyset$, $K_i^\Delta=\{0\}$ and $\frac{|K_i\setminus\{0\}|}{i}=0$ is even. If $\Phi(i)$ is nonempty, since $i=\lcm(\theta_{k_1},\cdots,\theta_{k_r})\notin S$, we have $r\ge 2$.

\begin{lem}\label{lem4.4}
 For $i\in T\setminus S$ with $\Phi(i)=\{\theta_{k_1},\cdots,\theta_{k_r}\}$($r\geq 2$), write $i = 2^l \cdot t$ with $t$ odd. Then:
\begin{enumerate}
    \item $\dfrac{|K_i^\Delta \setminus \{0\}|}{i} \equiv 0 \pmod{2}$;
    \item for each $j \in \{1,2,\dots,r\}$ with $v_2(\theta_{k_j}) =l$, we have $2^l \mid q^{u_{k_j}} - 1$, moreover
    \[
        \frac{1}{2^l}\bigl|K_{\theta_{k_j}}^\Delta \setminus \{0\}\bigr|
        \equiv \frac{1}{2^l}(q^{u_{k_j}} - 1) \pmod{2}.
    \]
\end{enumerate}
\end{lem}
\begin{proof}
Since $i\in T$, we have $i$ is the least common multiple of elements in $\Phi(i)$. Among $\theta_{k_1}, \ldots, \theta_{k_r}$, at least one achieves the maximum $2$-adic valuation $v_2(i)$. Without loss of generality, suppose
\[
v_2(\theta_{k_1}) = \cdots = v_2(\theta_{k_{r'}}) = v_2(i)
\quad \text{and} \quad
v_2(\theta_{k_{r'+j}}) < v_2(i) \text{ for } 1 \leq j \leq r - r',
\]
where $r' \geq 1$.

We prove~(1) and~(2) simultaneously by induction on $r$.

\medskip
\noindent\textit{Base case.} When $r = 2$, first we prove statement~(1). Since $i \notin S$, neither $\theta_{k_1} \mid \theta_{k_2}$ nor $\theta_{k_2} \mid \theta_{k_1}$. The assumption $\theta_j \nmid i$ for $j \notin \{k_1, k_2\}$ ensures $\theta_j \nmid \theta_{k_t}$ for $j \neq k_t$ ($t = 1, 2$), so that
\begin{align}\label{4.7}
|K_{\theta_{k_t}}^\Delta \setminus \{0\}|
= \sum_{j \mid \theta_{k_t}} |K_j^\Delta| - 1
= q^{u_{k_t}} - 1, \qquad t = 1, 2.
\end{align}
By the principle of inclusion-exclusion,
\begin{align}
|K_i^\Delta \setminus \{0\}|
&= \sum_{j \mid i} |K_j^\Delta|-|K_{\theta_{k_1}}^\Delta \setminus \{0\}|-|K_{\theta_{k_2}}^\Delta \setminus \{0\}|-1 \notag \\
&= q^{u_{k_1} + u_{k_2}}-(q^{u_{k_1}} - 1)-(q^{u_{k_2}} - 1) - 1 \notag \\
&= (q^{u_{k_1}} - 1)(q^{u_{k_2}} - 1). \label{eq:base-r2}
\end{align}
Since $v_2(\theta_{k_1}) = v_2(i)$,
\[
\frac{|K_i^\Delta \setminus \{0\}|}{i}\equiv \frac{(q^{u_{k_1}}-1)(q^{u_{k_2}}-1)}{2^{v_2(i)}}= \frac{|K_{\theta_{k_1}}^\Delta|}{ 2^{v_2(\theta_{k_1})}} \cdot (q^{u_{k_2}}-1).
\]
Since $\dfrac{|K_{\theta_{k_1}}^\Delta|}{\theta_{k_1}}=C_{\theta_{k_1}}^\Delta \in \mathbb{Z}$, we have $\theta_{k_1} \mid |K_{\theta_{k_1}}^\Delta|$ and in particular $2^{v_2(\theta_{k_1})} \mid |K_{\theta_{k_1}}^\Delta|$. Also $q$ is odd, and $q^{u_{k_2}} - 1$ is even. Hence
\[
2 \mid \frac{|K_{\theta_{k_1}}^\Delta|}{2^{v_2(\theta_{k_1})}} \cdot (q^{u_{k_2}}-1),
\]
so $\dfrac{|K_i^\Delta \setminus \{0\}|}{i}$ is even.

Now we prove statement~(2). For $j \in \{1, 2\}$ with $v_2(\theta_{k_j}) = l$, \eqref{4.7} gives
$|K_{\theta_{k_j}}^\Delta \setminus \{0\}| = q^{u_{k_j}} - 1$. Since $2^l \mid \theta_{k_j} \mid q^{u_{k_j}} - 1$, we have $2^l \mid q^{u_{k_j}} - 1$ and
\[
    \frac{1}{2^l}\bigl|K_{\theta_{k_j}}^\Delta \setminus \{0\}\bigr|
    = \frac{1}{2^l}(q^{u_{k_j}} - 1),
\]
so both sides have the same parity. This establishes~(2) for $r = 2$.

\medskip
\noindent\textit{Inductive step.}
Suppose both~(1) and~(2) hold for all integers in $T\setminus S$ with the image of $\Phi$ has size less than $r$. We first prove~(2), then use it to prove~(1).

To prove ~(2), it suffices to consider $j = 1$, if $v_2(\theta_{k_j})<v_2(i)$, then $\theta_{k_j}$ divides $m$ for some positive integer $m$ satisfying $|\Phi(m)|<|\Phi(i)|$. Hence, the conclusion follows from the inductive hypothesis. Suppose $\Phi(\theta_{k_1}) = \{\theta_{k_{i_1}}, \theta_{k_{i_2}}, \dots, \theta_{k_{i_a}}\}$ with $1 = i_1 < i_2 < \cdots < i_{a'} \leq r' < i_{a'+1} < \cdots < i_a$. We induct on $a$ (noting $a < r$).

When $a = 1$, the same argument as the base case gives $|K_{\theta_{k_1}}^\Delta \setminus \{0\}| = q^{u_{k_1}} - 1$, so $2^l \mid \theta_{k_1} \mid q^{u_{k_1}} - 1$ and $\frac{1}{2^l}|K_{\theta_{k_1}}^\Delta \setminus \{0\}| = \frac{1}{2^l}(q^{u_{k_1}} - 1)$, which have the same parity.

Suppose~(2) holds for $a - 1$ ($a \geq 2$). By dividing the divisors of $\theta_1$ into three classes, we have
\begin{align*}
    \frac{1}{2^l}\bigl|K_{\theta_{k_1}}^\Delta \setminus \{0\}\bigr|
    &= \frac{1}{2^l}\!\left(
        \sum_{j \mid \theta_{k_1}} |K_j^\Delta \setminus \{0\}|
        - \sum_{j \mid \theta_{k_1}/2} |K_j^\Delta \setminus \{0\}|
        - \sum_{\substack{2^l \mid j \mid \theta_{k_1}\,j\neq\theta_{k_1}}} |K_j^\Delta \setminus \{0\}|
    \right) \\
    &\equiv \frac{1}{2^l}\!\left(
        q^{u_{k_{i_1}}+\cdots+u_{k_{i_a}}}
        - q^{u_{k_{i_{a'+1}}}+\cdots+u_{k_{i_a}}}
        - \sum_{j=2}^{a'}\bigl|K_{\theta_{k_{i_j}}}^\Delta \setminus \{0\}\bigr|
    \right) \pmod{2},
\end{align*}
where the last congruence uses the induction hypothesis of~(1) (since $|\Phi(\theta_{k_{i_j}})| \leq a< r$ for $j \geq 2$). By the induction hypothesis on $a$, $\frac{1}{2^l}(q^{u_{k_{i_j}}}-1)$ is an integer and $\frac{1}{2^l}|K_{\theta_{k_{i_j}}}^\Delta \setminus \{0\}|\equiv \frac{1}{2^l}(q^{u_{k_{i_j}}}-1) \pmod{2}$ for $j = 2, \dots, a'$. Therefore,
\begin{align*}
    \frac{1}{2^l}\bigl|K_{\theta_{k_1}}^\Delta \setminus \{0\}\bigr|
    &\equiv \frac{1}{2^l}(q^{u_{k_{i_1}}+\cdots+u_{k_{i_{a'}}}}-1)
      - \sum_{j=2}^{a'} \frac{1}{2^l}(q^{u_{k_{i_j}}}-1) \\
    &\equiv \frac{1}{2^l}\!\left(
        (q^{u_{k_{i_1}}+\cdots+u_{k_{i_{a'-1}}}}-1)
        + q^{u_{k_{i_{a'}}}}-1
      \right)
      - \sum_{j=2}^{a'} \frac{1}{2^l}(q^{u_{k_{i_j}}}-1) \\
    &\equiv \cdots \\
    &\equiv \sum_{j=1}^{a'} \frac{1}{2^l}(q^{u_{k_{i_j}}}-1)
      - \sum_{j=2}^{a'} \frac{1}{2^l}(q^{u_{k_{i_j}}}-1) \\
    &\equiv \frac{1}{2^l}(q^{u_{k_1}}-1) \pmod{2}.
\end{align*}
Hence by induction on $a$, $2^l \mid |K_{\theta_{k_1}}^\Delta \setminus \{0\}|$ and $\frac{1}{2^l}|K_{\theta_{k_j}}^\Delta \setminus \{0\}|\equiv \frac{1}{2^l}(q^{u_{k_j}}-1) \pmod{2}$ for all $j \in [r']$.

\medskip
Now we look back to~(1). By dividing the divisors of $i$ into three classes,
\begin{align*}
    \frac{|K_i^\Delta \setminus \{0\}|}{i}
    &= \frac{1}{i}\!\left(
        \sum_{j \mid i} |K_j \setminus \{0\}|
        - \sum_{j \mid 2^{l-1}i} |K_j \setminus \{0\}|
        - \sum_{\substack{j = 2^l m,\, m \mid i,\, m \neq i}}
          |K_j \setminus \{0\}|
    \right) \\
    &= \frac{1}{i}\!\left(
        q^{u_{k_1}+\cdots+u_{k_r}}
        - q^{u_{k_{r'+1}}+\cdots+u_{k_r}}
        - \sum_{\substack{j = 2^l m,\, m \mid i,\, m \neq i}}
          |K_j \setminus \{0\}|
    \right),
\end{align*}
where when $l = 1$ the sum over $j \mid 2^{l-1}t$ is defined to be $0$. For $j = 2^l m$ with $m \mid t$, $m \neq t$: if $j \notin T$ then $K_j = \emptyset$; if $j \in T$ then $\Phi(j) \subsetneq \Phi(i)$, so $|\Phi(j)| < r$, and by the induction hypothesis on $r$, $\frac{|K_j \setminus \{0\}|}{j}$ is even. Hence
\begin{align*}
    \frac{|K_i \setminus \{0\}|}{i}
    &\equiv \frac{1}{i}\!\left(
        q^{u_{k_{r'+1}}+\cdots+u_{k_r}}
        (q^{u_{k_1}+\cdots+u_{k_{r'}}}-1)
        - \sum_{j=1}^{r'}\bigl|K_{\theta_{k_j}} \setminus \{0\}\bigr|
    \right)\\
&\equiv \frac{1}{2^l}\!\left(
        q^{u_{k_{r'+1}}+\cdots+u_{k_r}}
        (q^{u_{k_1}+\cdots+u_{k_{r'}}}-1)
        - \sum_{j=1}^{r'}\bigl|K_{\theta_{k_j}} \setminus \{0\}\bigr|
    \right) \pmod{2}.
\end{align*}
Applying~(2), we substitute
$|K_{\theta_{k_j}} \setminus \{0\}| \equiv q^{u_{k_j}} - 1 \pmod{2^{l+1}}$
for $j = 1, \dots, r'$, giving
\begin{align*}
    \frac{|K_i \setminus \{0\}|}{i}
    &\equiv \frac{1}{2^l}\!\left(
        q^{u_{k_{r'+1}}+\cdots+u_{k_r}}
        (q^{u_{k_1}+\cdots+u_{k_{r'}}}-1)
        - \sum_{j=1}^{r'}(q^{u_{k_j}}-1)
    \right) \pmod{2} \\
    &\equiv \frac{1}{2^l}\!\left(
        \prod_{j=1}^{r'}(q^{u_{k_j}}-1+1) - 1
        - \sum_{j=1}^{r'}(q^{u_{k_j}}-1)
    \right) \pmod{2}.
\end{align*}
Setting $x_j = q^{u_{k_j}} - 1$, then each $x_j$ is even with $2^l \mid x_j$.
\[
    \prod_{j=1}^{r'}(x_j + 1) - 1 - \sum_{j=1}^{r'} x_j
    = \sum_{\substack{A \subseteq [r'] \\ |A| \geq 2}} \prod_{j \in A} x_j,
\]
and each term in the sum is even and divisible by $2^{2l}$, so the whole expression is divisible by $2^{l+1}$. Therefore
\[\frac{1}{2^l}\!\left(
        \prod_{j=1}^{r'}(q^{u_{k_j}}-1+1) - 1
        - \sum_{j=1}^{r'}(q^{u_{k_j}}-1)
    \right) \equiv 0 \pmod{2},\]
and $\frac{|K_i \setminus \{0\}|}{i}$ is even. This completes the induction and the proof of both~(1) and~(2).
\end{proof}

 Combining Lemma~\ref{lem4.3} and Lemma~\ref{lem4.4}, the parity of $f$ equals the parity of
\[
  1+\left(\sum_{i=1}^{s}\frac{q^{u_i}-1}{2^{v_2(\theta_i)}}-1\right)
  \equiv
  \sum_{i=1}^{s}\frac{q^{u_i}-1}{2^{v_2(\theta_i)}}
  \pmod{2},
\]
which completes the proof of Theorem~\ref{thm:parity}.

\section{An algorithm for parity testing}\label{sec5}
Theorem~\ref{thm:parity} provides a practical algorithm to compute the parity of linearized permutation polynomials. In this section we outline algorithms for parity testing without considering larger fields avoids explicitly passing to larger fields.

To evaluate the sum $\sum_{i=1}^{s}\frac{q^{u_i}-1}{2^{v_2(\theta_i)}}$ appearing in Theorem~\ref{thm:parity}. we choose an odd integer $l$ and integer $t$ satisfies $2^t\cdot l$ divisible by all $\theta_i$ and exploit the identity.
\begin{lem}\label{lem5.1}
Assume that \(q\) is odd. Let \(t\geq 0\) and let \(l\) be a positive odd integer such that $\chi_{D_f}(X)\mid X^{2^t l}-1$, where $\chi_{D_f}$ denotes the characeteristic polynomialof $D_f$. Define
\[
  d_k := \deg\!\left(\gcd\!\left(\chi_{D_f},\; x^{2^k l}-1\right)\right),
  \qquad k = 0, 1, \dots, t.
\]
 Then
\begin{equation}\label{eq:parity-sum}
\sum_{i=1}^{s} \frac{q^{u_i}-1}{2^{v_2(\theta_i)}}
  \;\equiv\;
  \sum_{k=1}^{t} \frac{q^{d_k - d_{k-1}}-1}{2^k}
  \pmod{2}.
\end{equation}
\end{lem}
\begin{proof}
Grouping the terms in the sum by the $2$-adic valuation of $\theta_i$:
\begin{equation}\label{eq:sum-identity}
\begin{aligned}
  \sum_{i=1}^{s} \frac{q^{u_i}-1}{2^{v_2(\theta_i)}}
&=\sum_{j\geq 0}\sum_{v_2(\theta_i)=j} \frac{q^{u_i}-1}{2^j}\\
&\equiv \sum_{j\geq 0} \frac{q^{\sum_{v_2(\theta_i)=j}u_i}-1}{2^j}\\
&= q^{d_0}-1+\sum_{k=1}^{t} \frac{q^{d_k-d_{k-1}}-1}{2^k}\\
&\equiv\sum_{k=1}^{t} \frac{q^{d_k-d_{k-1}}-1}{2^k}\pmod 2.
\end{aligned}
\end{equation}
where the second congruence uses $\frac{q^a - 1}{2^j}+\frac{q^b-1}{2^j}\equiv \frac{q^{a+b}-1}{2^j}\pmod{2}$ for odd $q$, and the last step uses $q^{d_0}$ is odd.
\end{proof}
We present the algorithm in full below.

\begin{algorithm}\label{alg:parity-dickson}
\textit{Input:}
$f(x) = \sum_{i=0}^{n-1} a_i x^{q^i} \in \bF_{q^n}[x]$.

\textit{Output:} $f$ is not a PP, or $f$ is a PP together with its parity.
\begin{enumerate}
\item Using the coefficients $a_0, a_1, \dots, a_{n-1}$, construct the Dickson matrix $D_f$ as defined in~(\ref{1.1}).

\item Determine if $D_f$ is invertible. If $\det(D_f) = 0$, return that $f$ is \textbf{not} a PP,
and the algorithm terminates.

\item Decide whether $q$ is even.
\begin{itemize}
  \item If $q$ is even and $q = n = 2$, $f(x) = bx^2$ for some $b\in\bF_4^*$, return that $f$ is an \textit{odd} PP.
  \item If $q$ is even and $q \neq 2$ or $n \neq 2$ or $f(x)$ is not of the form $bx^2$ for any $b\in\bF_4^*$, return that $f$ is an \textit{even} PP.
\end{itemize}

\item Compute the characteristic polynomial
$\chi_{D_f} = \det(xE_n - D_f)$.

\item Compute the order $m$ of $\chi_{D_f}$,
i.e., the minimal positive integer $m$ such that $\chi_{D_f}(x) \mid x^m - 1$. Factor $m = 2^t \cdot l$, where $l$ is odd.

\item Compute the degrees
\[
  d_k := \deg\!\left(\gcd\!\left(\chi_{D_f},\; x^{2^k l}-1\right)\right),
  \quad k = 0, \dots, t.
\]

\item Compute the sum
\[
  S = \sum_{k=1}^{t} \frac{q^{d_k - d_{k-1}} - 1}{2^k}.
\]
If $S$ is odd, return that $f$ is an \textit{odd} PP; otherwise return that $f$ is an \textit{even} PP.
\end{enumerate}
\end{algorithm}

Now we present an example to illustrate Algorithm~\ref{alg:parity-dickson}.
\begin{exam}\label{exam5.1}
Let $q = n = 5$,  and $f(x) = 2x^{625} + x^{125} + 4x^{25} + x^5 + 4x$.
\medskip
The associated Dickson matrix is
$$
D_f =
\begin{pmatrix}
4^{5^0} & 1^{5^0} & 4^{5^0} & 1^{5^0} & 2^{5^0} \\
2^{5} & 4^{5} & 1^{5} & 4^{5} & 1^{5} \\
4^{5^2} & 2^{5^2} & 4^{5^2} & 1^{5^2} & 4^{5^2} \\
1^{5^3} & 4^{5^3} & 2^{5^3} & 4^{5^3} & 1^{5^3} \\
4^{5^4} & 1^{5^4} & 4^{5^4} & 2^{5^4} & 4^{5^4}
\end{pmatrix}
=
\begin{pmatrix}
4 & 1 & 4 & 1 & 2 \\
2 & 4 & 1 & 4 & 1 \\
1 & 2 & 4 & 1 & 4 \\
4 & 1 & 2 & 4 & 1 \\
1 & 4 & 1 & 2 & 4
\end{pmatrix},
$$
This matrix is invertible, so $f$ is a permutation polynomial of $\bF_{5^5}$. Since $q=5$ is odd, we determine the parity using Algorithm~\ref{alg:parity-dickson}. The characteristic polynomial is $\chi_{D_f}(x) = x^5 + 3$, with order $m = 20 = 2^2 \cdot 5$; thus $t=2$ and $l=5$. The degrees $d_1 = \deg\gcd(\chi_{D_f}, x^5-1) = 0$ and $d_2 = \deg\gcd(\chi_{D_f}, x^{20}-1) = 5$ give
$$
S = \frac{q^{d_1-d_0}-1}{2} + \frac{q^{d_2-d_1}-1}{4}
  = \frac{5^{0-0}-1}{2} + \frac{5^{5-0}-1}{4}
  = 781,
$$
which is odd. Hence $f$ is an odd permutation polynomial of $\bF_{5^5}$.

\end{exam}
When the linearized polynomial $f(x) = \sum_{i=0}^{n-1} a_i x^{q^i}$ has all coefficients in the base field $\bF_q$, the Dickson matrix $D_f$ is a polynomial in the companion matrix
$$J := \begin{pmatrix} 0 & E_{n-1} \\ 1 & 0 \end{pmatrix}$$ with coefficients $a_0, \dots, a_{n-1}$. Since the eigenvalues of $J$ are the $n$-th roots of unity $\omega_1, \dots, \omega_n$ in $\overline{\bF}_q$, the eigenvalues of $D_f$ are $g(\omega_i)$ for $i = 1, \dots, n$, where
$g(x) = \sum_{i=0}^{n-1} a_i x^i \in \bF_q[x]$ is the associated polynomial of $f$. Consequently, the characteristic polynomial $\chi_{D_f}$ and all subsequent GCD computations reduce to polynomial arithmetic in $\bF_q[x]/(x^n - 1)$, yielding the following simplified algorithm.

\begin{algorithm}\label{algorithm5.4}
\textit{Input:}
$f(x) = \sum_{i=0}^{n-1} a_i x^{q^i} \in \bF_q[x]$,
a linearized polynomial over $\bF_{q^n}$.

\textit{Output:} $f$ is not a PP, or $f$ is a PP together with its parity.

\begin{enumerate}
    \item Construct the associate polynomial
    $g(x) = a_0 + a_1 x + \cdots + a_{n-1} x^{n-1} \in \bF_q[x]$.

    \item Determine if $g(x)$ and $x^n - 1$ are coprime.
    If $\gcd(g(x),\, x^n - 1) \neq 1$, return $0$ and the algorithm terminates
    (i.e., $f$ is not a PP).

    \item Decide whether $q$ is even.
    \begin{itemize}
        \item If $q$ is even and $q = n = 2$, $f(x) = bx^2$ for some $b\in\bF_4^*$,
              return that $f$ is an \textit{odd} PP.
        \item If $q$ is even and $q \neq 2$, $n \neq 2$ or $f(x)$ is not of the form $bx^2$ for any $b\in\bF_4^*$,
              return that $f$ is an \textit{even} PP.
    \end{itemize}

    \item Compute the minimal positive integer $m$ such that
    $x^n - 1 \mid g(x)^m - 1$.
    Factor $m = 2^t \cdot l$, where $l$ is odd.

    \item Compute the degrees
    $$
        d_k := \deg\!\left(\gcd\!\left(x^n - 1,\; g(x)^{2^k l} - 1\right)\right),
        \quad k = 0, \dots, t.
    $$

    \item Compute the sum
    $$
        S = \sum_{k=1}^{t} \frac{q^{d_k - d_{k-1}} - 1}{2^k}.
    $$
    If $S$ is odd, return that $f$ is an \textit{odd} PP; otherwise return that $f$ is an \textit{even} PP.
\end{enumerate}
\end{algorithm}

Now we test the parity of $f(x)$ in Example~\ref{exam5.1} again for illustrating Algorithm~\ref{algorithm5.4}
\begin{exam}\label{exam5.2}
Let $q=5$ and $f(x) = 4x + x^5 + 4x^{25} + x^{125} + 2x^{625} \in \bF_5[x]$. We test the parity of $f$ using Algorithm~\ref{algorithm5.4}. The associated polynomial is $g(x) = 4 + x + 4x^2 + x^3 + 2x^4 \in \bF_5[x]$, and $\gcd(g(x), x^5-1) = 1$, so $f$ is a permutation polynomial of $\bF_{5^5}$. Since $q=5$ is odd, we proceed. The minimal positive integer $m$ satisfying $x^5-1 \mid g(x)^m-1$ in $\bF_5[x]/(x^5-1)$ is $m=20=2^2\cdot 5$, so $t=2$ and $l=5$. Computing $d_1 = \deg\gcd(x^5-1, g(x)^{10}-1) = 0$ and $d_2 = \deg\gcd(x^5-1, g(x)^{20}-1) = 5$, we obtain
$$
S = \frac{5^{d_1-d_0}-1}{2} + \frac{5^{d_2-d_1}-1}{4}
  = \frac{5^0-1}{2} + \frac{5^5-1}{4}
  = 0 + 781 = 781,
$$
which is odd. Hence $f$ is an odd permutation polynomial of $\bF_{5^5}$, consistent with the result obtained in Example~\ref{exam5.1}
\end{exam}
\begin{rem}
   When the coefficients of $f$ are restricted to $\bF_q$, Algorithm~\ref{algorithm5.4} can be recovered relatively easily from the cycle structure analysis of linearized permutation polynomials over $\bF_{q^n}$ established by Mullen and Vaughan~\cite{MR959697} and may thus be viewed as a direct application of their results.
\end{rem}
\section{Conclusion}\label{sec6}
In this paper, we present results on the cycle structure and parity of linearized permutation polynomials over finite fields using Dickson matrices. Building on the cycle structure, Theorem~\ref{thm:parity} gives a complete parity criterion for a linearized permutation polynomial. In characteristic two, a $q$-linearized permutation polynomial over $\bF_{q^n}$ is odd precisely when $q=n=2$ and $f(x)=bx^2$ with $b\in\bF_4^*$. In odd characteristic, its parity is determined by the multiplicative orders of the eigenvalues of the Dickson matrix and their total algebraic multiplicities. This criterion also leads to an algorithm for the parity testing.
It would be interesting to extend the method in this work to other class of permutation polynomials, as the work on Dickson polynomials~\cite{MR1159877}.

\section{Declaration on the Use of AI}

The results of this paper, including all mathematical derivations and proofs, were conceived, developed, and verified by the authors. During the preparation of this manuscript, AI-assisted tools were used to help draft the Introduction and to improve the clarity and readability of the language throughout the paper. No mathematical reasoning, proof steps, or conclusions were generated by these tools. All changes suggested with AI assistance were reviewed by the authors, who take full responsibility for the content of this work.

\bibliographystyle{plain}

\bibliography{ParityPP}

\end{document}